\documentclass[12pt]{amsart}

\usepackage{amsmath,amssymb,amsthm,
	amsfonts}

	\newtheorem{dfn}{Definition}[section]
	
	\newtheorem{prop}[dfn]{Proposition}
	
	\newtheorem{lem}[dfn]{Lemma}
	
	\newtheorem{ex}[dfn]{Example}
	\newtheorem{claim}[dfn]{Claim}

	\newtheorem{ack}{Acknowledgements\!\!}
	\newtheorem{mo}[ack]{Exercise\!\!}

	\numberwithin{equation}{section}

	\newcommand{\dist}{\mathop{\mathit{d}} \nolimits}
	
	\newcommand{\diam}{\mathop{\mathrm{diam}} \nolimits}
	\newcommand{\ric}{\mathop{\mathit{Ric}}  \nolimits}
	
	\newcommand{\card}{\mathop{\mathrm{Card}} \nolimits}

	\newcommand{\lip}{\mathop{\mathcal{L}ip}      \nolimits}
	\newcommand{\me}{\mathop{\mathrm{me}}      \nolimits}
	\newcommand{\supp}{\mathop{\mathrm{Supp}}    \nolimits}
	\newcommand{\rlip}{\mathop{\mathrm{Lip}}     \nolimits}
	\newcommand{\oblip}{\mathop{\stackrel{\rlip_1}{\longrightarrow}}    \nolimits}
	
	\newcommand{\obd}{\mathop{\underline{H}_{\lambda}\mathcal{L}\iota_1}    \nolimits}
	
	\newcommand{\bobd}{\mathop{H_{\lambda}\mathcal{L}\iota_1}
	\nolimits} 
	
	\newcommand{\sikaku}{\mathop{\underline{\square}_{\lambda}}                 \nolimits}
	\newcommand{\sikakuu}{\mathop{\underline{\square}_{1}}                 \nolimits}

	\newcommand{\bounasi}{\mathop{\square_{\lambda}}    \nolimits}
	\newcommand{\vol}{\mathop{\mathit{vol}}        \nolimits}

\begin{document}

	\title[Estimates of Gromov's box distance]
	{Estimates of Gromov's box distance}
	\author[Kei Funano]{Kei Funano}
	\address{Mathematical Institute, Tohoku University, Sendai 980-8578, JAPAN}
	\email{sa4m23@math.tohoku.ac.jp}
	\dedicatory{}
	\date{\today}

	\maketitle

	\setlength{\baselineskip}{5mm}

	\begin{abstract}In $1999$, M. Gromov introduced the box distance
	 function $\sikaku$ on the space of all mm-spaces. In this paper, by using the method of
	 T. H. Colding (cf.~\cite[Lemma $5.10$]{Colding}), we estimate $\sikaku
	 (\mathbb{S}^n,\mathbb{S}^m)$ and $\sikaku (\mathbb{C}P^n,
	 \mathbb{C}P^m)$, where $\mathbb{S}^n$ is the $n$-dimensional
	 unit sphere in
	 $\mathbb{R}^{n+1}$ and $\mathbb{C}P^n$ is the $n$-dimensional complex projective space equipped with the Fubini-Study metric. In paticular,
	 we give the complete answer to an Exercise of Gromov's
	Green book (cf.~\cite[Section $3\frac{1}{2}.18$]{gromov}). We
	 also estimate $\sikaku \big(SO(n), SO(m)\big)$ from
	 below, where $SO(n)$ is the special orthogonal group. 
	\end{abstract}

	\section{Introduction}
	In $1999$, M. Gromov developed the theory of mm-spaces in
	\cite[Chapter $3\frac{1}{2}_{+}$]{gromov} by introducing two
	distance functions, called the \emph{box distance function}
	$\sikaku$ and the \emph{observable distance function} $\obd$, on the space $\mathcal{X}$ of all
	isomorphic class of mm-spaces. Here, an \emph{mm-space} is a
	triple $(X, \dist_X, \mu_X)$, where $\dist_X$ is a complete separable
	metric on a set $X$ and $\mu_X$ a finite Borel measure on $(X,
	\dist_X)$. 
	The notion of the distance function $\sikaku$ is considered as a
	natural extension of the Gromov-Hausdorff
	distance function to the space $\mathcal{X}$. On the other hand, the notion
	of the distance function $\obd$ is related to measure
	concentration. Roughly speaking, ``measure concentration''
	amounts to saying that the push-forward measures
	$f_{n \ast}(\mu_n)$ on $\mathbb{R}$ concentrate to a point
	for any sequence of $1$-Lipschitz functions $f_n:(X_n, \dist_n ,
	\mu_n)\to \mathbb{R}$. For instance, the unit
	spheres in Euclidean spaces $\{ \mathbb{S}^n
	\}_{n=1}^{\infty}$, the complex projective spaces $\{
	\mathbb{C}P^n   \}_{n=1}^{\infty}$ equipped with the
	Fubini-Study metrics, and the special orthogonal
	groups $\{
	SO(n)\}_{n=1}^{\infty}$ have that property. He defined the
	distance $\obd(X,Y)$ by using the
	Hausdorff distance between the space of $1$-Lipschitz
	functions on $X$ and that on $Y$, and showed that a
	sequence $\{  X_n    \}_{n=1}^{\infty}$ of mm-spaces
	concentrates if and only if the sequence $\{  X_n   \}_{n=1}^{\infty}$
	converges to a one-point space with respect to the distance
	function $\obd$.

	The topology on $\mathcal{X}$
	determined by $\sikaku$ is strictly stronger than that of
	$\obd$. In fact, the sequences $\{ \mathbb{S}^n
	\}_{n=1}^{\infty}$, $\{\mathbb{C}P^n\}_{n=1}^{\infty}$, and $\{
	SO(n)\}_{n=1}^{\infty}$ are all divergent with respect to the
	distance $\sikaku$ (see Proposition \ref{kyou}). This is related
	to the following exercise in Gromov's book:
	\begin{mo}[{cf.~\cite[Section
	 $3\frac{1}{2}.18$]{gromov}}]Estimate the distance $\sikaku (\mathbb{S}^n,
	 \mathbb{S}^m)$.
	\end{mo}
	To solve the exercise, applying a method of \cite[Lemma $5.10$]{Colding}, we will estimate $\sikaku(M,N)$ from below for compact Riemannian
	manifolds $M$ and $N$ with positive Ricci curvatures and the volume measures
	satisfying a homogenuity condition (see Lemma \ref{hyouka}). As a result,
	we get the following proposition:
	\begin{prop}\label{kaotan}Assume that two sequences $\{ n_k
	   \}_{k=1}^{\infty}$, $\{ m_k  \}_{k=1}^{\infty}$ of natural
	   numbers satisfy $n_k \leq C_1 k, m_k\leq C_2 k$ and $|n_k
	   -m_k|\geq C_3 k$, $k=1,2, \cdots$, for some positive
	   constants $C_1,C_2,C_3$. Then, we have
	   \begin{align*}
	    \liminf_{k \to \infty} \sikakuu (\mathbb{S}^{n_k},
	    \mathbb{S}^{m_k}),  \liminf_{k \to \infty} \sikakuu
	    (\mathbb{C}P^{n_k}, \mathbb{C}P^{m_k})\geq \  \min \Big\{
	    2^{-\frac{C_1}{C_3}}\pi^{- \frac{C_2}{C_3}},
	    2^{-\frac{C_2}{C_3}}\pi^{-\frac{C_1}{C_3}}  \Big\}.   
	    \end{align*}In paticular, if in addition $|n_k-m_k|\geq C_4
	   k^{\alpha},\ k=1,2, \cdots$, holds for some constant $C_4>0$
	   and a number $\alpha > 1$, then we have
	   \begin{align*}
	    \lim_{k\to \infty}\sikakuu (\mathbb{S}^{n_k},
	    \mathbb{S}^{m_k}), \lim_{k\to \infty}\sikakuu
	    (\mathbb{C}P^{n_k}, \mathbb{C}P^{m_k})=1.
	    \end{align*}
	 \end{prop}

	 We estimate $\sikaku \big(
	 SO(n), SO(m)\big)$ from below by the differnce of
	 their diameters (see Lemma \ref{peach}). Consequently, we
	 obtain the following proposition:
 \begin{prop}\label{oosawaakane}Assume that two sequences $\{
	    n_k \}_{k=1}^{\infty}$, $\{ m_k   \}_{k=1}^{\infty}$ of
	   natural numbers satisfy 
	   $n_k \leq C_1 k$, $m_k \leq C_2 k$ and $|n_k -m_k|\geq C_3
	   \sqrt{k}$, $k=1,2, \cdots$, for some positive constants $C_1, C_2 , C_3$. Then, we have
	   \begin{align*}
	    \liminf_{k \to \infty} \sikakuu \big( SO(n_k) , SO(m_k)
	    \big)\geq
	    \min \Big\{  \frac{1}{2}, \frac{C_3}{\sqrt{C_1}     +
	    \sqrt{C_2} }    \Big\}.
	    \end{align*}In paticular, if in addition $|n_k - m_k|\geq C_4
	   k^{\alpha}$, $k=1,2,\cdots$, holds for some constant $C_4 >0$ and a number
	   $\alpha > 1/2$, then we have 
	   \begin{align*}
	    \liminf_{k \to \infty}  \sikakuu \big( SO(n_k), SO(m_k)
	    \big)\geq \frac{1}{2}.
	    \end{align*}
\end{prop}

As is related to the above Gromov's exercise, we also proves the
following proposition. This proposition is also mentioned by Gromov in
\cite[Section $3\frac{1}{2}.3$ Exercise $(e)$]{gromov}.
\begin{prop}\label{azusa}
We have
	    \begin{align*}
	     \sikakuu (\mathbb{S}^n, \mathbb{S}^{n-1}), \sikakuu
	     (\mathbb{C}P^n, \mathbb{C}P^{n-1})\to 0
	     \end{align*}as $n \to \infty$.
\end{prop}

\section{Preliminaries}

  \subsection{Definition of Gromov's box distance function $\sikaku$}
	\begin{dfn}\upshape
	 Let $\lambda \geq 0$ and $(X,\mu)$ be a measure
	   space with $\mu(X)< +\infty$. For two maps $\dist_1,\dist_2:X\times X \to
	   \mathbb{R}$, we define a number $\bounasi (\dist_1,\dist_2)$
	   as the infimum of $\varepsilon >0$ such that there exists a
	   measurable subset $T_{\varepsilon}\subseteq X$ of measure
	   at least $\mu(X)-\lambda \varepsilon$ satisfying
	   $|\dist_1(x,y)-\dist_2(x,y)|\leq \varepsilon$ for any $x,y
	   \in T_{\varepsilon}$.
	 \end{dfn}
	   It is easy to see that this is a distance function on the set of all
	   functions on $X\times X$, and the two distance functions $\square_{\lambda}$ and
	   $\square_{\lambda'}$ are equivalent to each other for any
	   $\lambda,\lambda'>0$. 
	   \begin{dfn}[parameter]\upshape Let $X$ be an mm-space and
	    $\mu(X)=m$. Then, there exists a Borel measurable map
	    $\varphi :[0,m]\to X$ with $\varphi
	    _{\ast}(\mathcal{L})=\mu$, where $\mathcal{L}$ stands for
	    the Lebesgue measure on $[0,m]$. We call $\varphi$ a
	    \emph{parameter} of $X$. 
	    \end{dfn}
	    Note that if the support of $X$ is not a one-point, then its
	    parameter is not unique.
	   \begin{dfn}[Gromov's box distance function]\upshape If two mm-spaces $X,Y$ satisfy $\mu_X(X)=\mu_Y(Y)=m$, we define
		\begin{align*}
		 \sikaku (X,Y):= \inf \square_{\lambda}(\varphi_X^{\ast}\dist_X,
\varphi_Y^{\ast}\dist_Y),
\end{align*}where the infimum is taken over all parameters 
$\varphi_X:[0,m] \to X, \ \varphi_Y :[0,m] \to Y$, and
$\varphi_X^{\ast}\dist_X$ is defined by $\varphi_X^{\ast}\dist_X(s,t):=\dist_X(\varphi_X(s),\varphi_X(t))$ for $s,t \in [0,m]$.
If $\mu_X(X)<\mu_Y(Y)$, putting $m:=\mu_X(X), m':=\mu_Y(Y)$
, we define
\begin{align*}
\sikaku (X,Y):= \sikaku \Big(  X, \frac{m}{m'}Y \Big)+m'-m, 
\end{align*}where $(m/m')Y:=(Y,\dist_Y,(m/m')\mu_Y)$. 
\end{dfn}

We recall that two mm-spaces are \emph{isomorphic} to each other if
there is a measure preserving isometry between the supports of their
measures. $\sikaku$ is a distance function on $\mathcal{X}$ for any
$\lambda \geq 0$. See \cite[Section 1, 3]{funafuna} for a complete proof
of that. Note that the distance functions $\sikaku$ and $\underline{\square}_{\lambda'}$ are
equivalent to each other for distinct $\lambda , \lambda'>0$. 

\subsection{Definition of observable distance functions $\obd$}
For a measure space $(X,\mu)$ with $\mu(X)< +\infty$, we denote by
$\mathcal{F}(X, \mathbb{R})$ the space of all functions on $X$. Given 
$\lambda \geq 0$ and $f, g \in \mathcal{F}(X, \mathbb{R})$, we put 
\begin{align*}
 \me_{\lambda}(f,g):= \inf \{  \varepsilon >0 \mid \mu 
 \big(\{ x\in X \mid |f(x)-g(x)|\geq \varepsilon \} \big)\leq \lambda \varepsilon             \}.
\end{align*}Note that this $\me_{\lambda}$ is a distance function on $\mathcal{F}(X,
 \mathbb{R})$ for any $\lambda \geq 0$ and its topology on
 $\mathcal{F}(X, \mathbb{R})$ coincides with the topology of the 
convergence in measure for any $\lambda >0$. Also, the distance functions
 $\me_{\lambda}$ for all $\lambda >0$ are mutually equivalent. 

 We recall that the \emph{Hausdorff distance} between two closed
subsets $A$ and $B$ in a metric space $X$ is defined by
 \begin{align*}
\dist_H (A,B):= \inf \{  \varepsilon >0 \mid A \subseteq B_{\varepsilon} ,
  B \subseteq A_{\varepsilon}       \},
 \end{align*}where $A_{\varepsilon}$ is a closed $\varepsilon$-neighborhood of $A$.

Let $(X,\mu)$ be a measure space with $\mu(X)< +\infty$. For a
semi-distance $\dist$ on $X$, we indicate by $\lip_1 (\dist)$ the space of
all $1$-Lipschitz functions on $X$ with respect to $\dist$. Note that
$\lip_1 (\dist)$ is a closed subset in
$(\mathcal{F}(X,\mathbb{R}),\me_{\lambda})$ for any $\lambda\geq 0$.

 \begin{dfn}\upshape For $\lambda \geq 0$ and two semi-distance functions $\dist , \dist'$ on $X$, we define
  \begin{align*}
   \bobd (\dist ,\dist'):= \dist_{H} \big(  \lip_1(\dist ), \lip_1 (\dist')      \big),
   \end{align*}where $\dist_H$ stands for the Hausdorff distance
  function in $(\mathcal{F}(X,\mathbb{R}),\me_{\lambda}).$
  \end{dfn}
This $\bobd$ is actually a distance function on the space of all
semi-distance
functions on
$X$ for all $\lambda \geq 0$, and the two distance functions $\bobd$ and
$H_{\lambda'}\mathcal{L}\iota_{1}$ are equivalent to each other for any
$\lambda, \lambda' >0$.
  \begin{lem}\label{mikitii}For any two semi-distance functions $\dist , \dist'$ on $X$, we have
   \begin{align*}
    \bobd (\dist , \dist')\leq \bounasi (\dist, \dist').
    \end{align*}
   \begin{proof}For any $\varepsilon>0$ with $\bounasi (X,Y)<
    \varepsilon$, there exists a measurable subset $T_{\varepsilon} \subseteq
    X$ such that $\mu (X\setminus T_{\varepsilon})\leq \lambda
    \varepsilon$ and $|\dist(x,y)-\dist'(x,y)|\leq \varepsilon$ for any
    $x,y \in T_{\varepsilon}$. Given arbitrary $f\in \lip_1(\dist)$, we
    define $\widetilde{f}\in \mathcal{F}(X, \mathbb{R})$ by $\widetilde{f}(x):= \inf
    \{ f(y)+\dist' (x,y) \mid y\in T_{\varepsilon}  \}$. We see
    easily that $\widetilde{f} \in \lip_1 (\dist')$ and $\widetilde{f}(x)
    \leq f(x)$ for any $x\in T_{\varepsilon}$. Taking any $x\in
    T_{\varepsilon}$, we have
    \begin{align*}
     |f(x)- \widetilde{f}(x)| = \ &f(x)-\widetilde{f}(x)\\ = \ &  \sup
     \{ f(x)-f(y)-d'(x,y) \mid 
     y\in T_{\varepsilon} \}\\  \leq \ &\sup \{
     d(x,y)-d'(x,y) \mid y\in T_{\varepsilon}  \} \\ \leq  \
     &\varepsilon.  
     \end{align*}Therefore, we get $\me_{\lambda}(f, \widetilde{f})\leq
    \varepsilon $, which implies $\lip_1 (\dist) \subseteq \big( \lip_1
    (\dist')      \big)_{\varepsilon}$. Similary, we also have
    $\lip_1(\dist') \subseteq \big( \lip_1 (\dist)
    \big)_{\varepsilon}$, which yields $\bobd (\dist , \dist')\leq
    \varepsilon$. This completes the proof.
    \end{proof}
   \end{lem}

  \begin{dfn}[Observable distance function]If two mm-spaces $X,Y$ satisfy $\mu_X
   (X) =\mu_Y(Y)=m$, we define
   \begin{align*}
    \obd (X,Y):= \inf \bobd (\varphi_{X}^{\ast}\dist_X, \varphi_{Y}^{\ast}\dist_Y),
   \end{align*}where the infimum is taken over all parameters
   $\varphi_X :[0,m]\to X, \ \varphi_{Y}:[0.m]\to Y$. If $\mu_X (X) <
   \mu_Y (Y)$, putting $m:= \mu_X (X), m':= \mu_Y (Y)$, we define 
   \begin{align*}
    \obd (X,Y):= \obd \Big( X, \frac{m}{m'}Y \Big) + m'-m.
    \end{align*}
  \end{dfn}
$\obd$ is a distance function on $\mathcal{X}$ for any $\lambda\geq 0$.
See \cite[Section $3$]{funafuna} for a complete proof of that. Note that the distance functions $\obd$ and
$\underline{H}_{\lambda'}\mathcal{L}\iota_1$ are equivalent to each
other for any $\lambda, \lambda' >0$.

For a Borel measure $\nu$ on $\mathbb{R}$ with $m:=\nu
(\mathbb{R})<+\infty$ and $\kappa >0$, we put
\begin{align*}
\diam (\nu,m-\kappa):=  \inf \{ \diam Y \mid Y \subseteq \mathbb{R} \text{ is a Borel subset such that }\nu_Y(Y)\geq m-\kappa\},
\end{align*}and call it the \emph{partial diameter} on $\nu$.

\begin{dfn}[Observable diameter]\upshape Let $(X,\dist,\mu)$ be an
 mm-space and let $m:=\mu(X)$. For any $\kappa >0$ we
	 define the \emph{observable diameter} of $X$ by 
	 \begin{align*}
	  \diam (X\oblip \mathbb{R}, m-\kappa):=
	   \sup \{ \diam (f_{\ast}(\mu),m-\kappa) \mid f:X\to \mathbb{R} \text{ is an }1 \text{{\rm -Lipschitz function}}  \}.
	  \end{align*}
\end{dfn}

The idea of the observable diameter came from the quantum and
statistical mechanics, that is, we think of $\mu$ as a state on a
configuration space $X$ and $f$ is interpreted as an observable. We define a sequence $\{  X_n   \}_{n=1}^{\infty}$ of
	mm-spaces is a \emph{L\'{e}vy family} if $\diam (X_n\oblip \mathbb{R},m_n-\kappa)\to 0$ as $n\to \infty$ for any $\kappa >0$, where
	$m_n$ is the total measure of the mm-space $X_n$. This is equivalent to that for any
	$\varepsilon >0$ and any sequence $\{ f_n :X_n \to \mathbb{R}\}_{n=1}^{\infty}$ of $1$-Lipschitz functions, we have 
	\begin{align*}
	 \mu_n(\{ x\in X_n \mid |f_n(x) - m_{f_n}|\geq \varepsilon
	 \})\to 0 \tag{$\diamondsuit$}  \text{ as } n\to \infty,
	\end{align*}where $m_{f_n}$ is a some constant
	determined by $f_n$. 

	\begin{ex}\upshape Let $\{ M_n   \}_{n=1}^{\infty}$ be a sequence of
	 compact connected Riemannian manifolds. Let $\dist_n$ be the
	 Riemannian distance on $M_n$ and $\mu_n$ be its Riemannian
	 volume measure normalized as $\mu_n(M_n)=1$. Assume that $\ric_{M_n}\geq \kappa_n \to
	 +\infty$ as $n\to \infty$. Then, by virtue of L\'{e}vy-Gromov's isoperimetric
	 inequality, the sequence $\{ M_n   \}_{n=1}^{\infty}$ is a
	 L\'{e}vy family ({cf.~\cite[Section $1$, Remark
	 $2$]{milgro}}). For example, $ \{ \mathbb{S}^n \}_{n=1}^{\infty}
	 $ and $ \{ \mathbb{C}P^n\}_{n=1}^{\infty}$
	 are L\'{e}vy
	 families. Recall that the Fubini-Study metric on
	 $\mathbb{C}P^n$ is the unique
	 Riemannian metric on $\mathbb{C}P^n$ such that the canonical
	 projection $\mathbb{S}^{2n+1} \to \mathbb{C}P^n$ is a Riemannian
	 submersion. Since $\ric_{SO(n)}\geq (n-1)/4$, the sequence $\{
	 SO(n)\}_{n=1}^{\infty}$ is L\'{e}vy family. Since the distance
	 function 
	 induced from the Hilbert-Schmidt norm on $SO(n)$ is not greater
	 than that of the Riemannian distance function, $\{  SO(n)
	 \}_{n=1}^{\infty}$ is L\'{e}vy family with respect to also the
	 Hilbert-Schmidt norms. 
	 \end{ex}

	\begin{ex}[Hamming cube]\upshape Let $\mu_n$ be the normalized counting
	 measure on $\{   0,1  \}^n$ and $\dist_n$ be the \emph{Hamming
	 distance function} on $\{  0,1    \}^n$, that is,  
	 \begin{align*}
	  \dist_n \big( (x_i)_{i=1}^n,   (y_i)_{i=1}^n\big):= \frac{1}{n}\card \big( \{  i\in \{1, \cdots
	  ,n\} \mid x_i \neq y_i       \} \big).
	  \end{align*}The mm-space $\{ 0,1 \}^n$ is called the
	 \emph{Hamming cube}. The sequence $\big\{  \{ 0,1   \}^n
	 \big\}_{n=1}^{\infty}$ is a L\'{e}vy family (cf.~\cite[Section $3\frac{1}{2}.42$]{gromov}).
	 \end{ex}

Gromov showed the following proposition by considering a constant $m_{f_n} $ in
$(\diamondsuit)$ as a Lipschitz funtion from a one-point space $\{ \ast_n
\}$ with total measure $\mu_n(X_n)$. 
\begin{prop}\cite[Section $3\frac{1}{2}.45$]{gromov}\label{onepoint} A sequence $\{  X_n   \}_{n=1}^{\infty}$ of mm-spaces
 is a L\'{e}vy family if and only if $\obd (X_n , \{  \ast_n  \})\to 0$
 as $n\to \infty$ for any $\lambda >0$.
\end{prop}
	 \section{Estimates of Gromov's box distance function}

Let $X$ be a metric space. Denote by $B_X(x,r)$ the closed ball in $X$
centered at $x\in X$ with radius $r>0$. A Borel measure $\mu$ on $X$ is said to be
\emph{uniformly distributed} if  
\begin{align*}
0< \mu\big( B_X(x,r) \big)= \mu\big(B_X(y,r)\big)<+ \infty
\end{align*}for any $r>0$ and $x,y \in X$.

From Lemma \ref{mikitii}, we see that the topology on $\mathcal{X} $ determined by $\sikaku$
is not weaker than that of $\obd$ for any $\lambda \geq 0$. For a Borel
measure $\mu$ on a metric space, we denote by $\supp \mu$ its support.
	 \begin{prop}\label{kyou}Let $\{ (X_n ,\dist_n, \mu_n)   \}_{n=1}^{\infty}$
	  be a L\'{e}vy family such that $\mu_n $ is uniformly
	  distributed Borel probability measure satisfying $X_n = \supp
	  \mu_n$ and
	  $\inf\limits_{n\in \mathbb{N}}\diam X_n >0$. Then, the
	  sequence $\{ X_n
	  \}_{n=1}^{\infty}$ does not converge with respect to
	  the distance function $\sikaku$ for any $\lambda \geq 0$.
	  \begin{proof}
	   Suppose that $\{  X_n  \}_{n=1}^{\infty}$
	   convereges and let $X$ be its limit. Since $\{ X_n
	  \}_{n=1}^{\infty}$ is a L\'{e}vy family, by using
	   Proposition \ref{onepoint}, $X$ must be a one-point
	   space. Fix $\varepsilon >0$ with
	   $\varepsilon < \min \{  3, \inf\limits_{n\in \mathbb{N}} \diam
	   X_n \}/3$. For any suffieciently large $n\in
	   \mathbb{N}$, there exists a parameter $\varphi_n :[0,1]\to
	   X_n$ of $X_n$ and Borel subset $T_n \subseteq [0,1]$ such
	   that $\mathcal{L}(T_n)>1-\varepsilon /2$ and $\dist_n
	   \big(\varphi_n (s), \varphi_n (t)   \big)< \varepsilon /2$
	   for any $s,t \in T_n$. Fix a point $t_n \in T_n$.
	   There exists a point $x_n \in X_n$ such that $\dist_n \big(
	   \varphi_n(t_n), x_n    \big) \geq \diam X_n /3 > \varepsilon
	   $ and hence $B_{X_n} ( \varphi_n (t_n) ,\varepsilon
	   /2        ) \cap B_{X_n} ( x_n , \varepsilon
	   /2)=\emptyset$. Therefore, we get
	   \begin{align*}
	    1\geq \ &\mu_n \big( B_{X_n}(\varphi_n (t_n), \varepsilon /2)
	    \cup  B_{X_n} (x_n , \varepsilon /2)             \big)\\ 
	    = \ &    2\mu_n \big( B_{X_n} (\varphi_n (t_n), \varepsilon /2)
	    \big)\\ = \ &2\mathcal{L}\big( \varphi_n^{-1} \big(
	    B_{X_n}(\varphi_n(t_n), \varepsilon /2)         \big)
	    \big)\geq 2\mathcal{L}(T_n) \geq 2-\varepsilon  > 1,
	   \end{align*}which gives a contradiction. This completes the proof.
	  \end{proof}
	  \end{prop}

	  From Proposition \ref{kyou}, we see that many L\'{e}vy
	  families such as $\{ \mathbb{S}^n\}_{n=1}^{\infty}$, $\{
	  \mathbb{C}P^n\}_{n=1}^{\infty} $, $\{ SO(n)
	  \}_{n=1}^{\infty}$, and $\big\{  \{ 0,1   \}^n
	  \big\}_{n=1}^{\infty} $ have no convergent subsequences with
	  respect to the distance function $\sikaku$. Therefore, the
	  distance function
	  $\sikaku$ determines the topology on $\mathcal{X}$ strictly
	  stronger than that of the distance function $\obd$ for any $\lambda >0$.
	  However, since the proof of Proposition \ref{kyou} is by
	  contradition, we do not estimate $\sikaku
	  (X_n,X_m)$ from below for $n,m\in \mathbb{N}$.

	  The proof of the following proposition is an analogue of
	  the proof of \cite[Lemma 5.10]{Colding}
	 \begin{lem}\label{ballnomeasure}Let $(X, \dist_X , \mu_X),
	  (Y,\dist_Y, \mu_Y)$ be mm-spaces and assume that $\mu_X ,
	  \mu_Y$ are uniformly distributed Borel probability
	  measures. Denote by $v_X (r)$ $($respectively, $v_Y (r)$$)$
	  the measure of a closed ball of $X$ $($respectively, $Y$$)$
	  with radius $r>0$ and assume that $v_X(a+c)\leq (1-c) v_Y
	  (a/2)$ for some $a,c>0$ with $c<1$. Then, we have $\sikakuu (X,Y)\geq
	  c$.
	  \begin{proof}Let us prove the lemma by
	   contradiction. Suppose that $\sikakuu (X,Y) < c$. Then,
	   there exist compact subset $T \subseteq [0,1]$ and two
	   parameters $\varphi_X :[0,1]\to X$, $\varphi_Y :[0,1]\to Y$
	   such that 
	   \begin{itemize}
	    \item[$(1)$]$\mathcal{L} (T)> 1-c$,
	    \item[$(2)$]$\varphi_X |_{T}:T \to X$, $\varphi_Y |_{T}:T
			 \to Y$ are continuous,
	    \item[$(3)$]$|\dist_X \big( \varphi_X (s), \varphi_X (t)
			\big)- \dist_Y     \big( \varphi_Y(s),
			\varphi_Y(t)        \big)|< c$ for any $s,t \in T$.
	    \end{itemize}By $(1)$ and $(2)$, $\varphi_Y(T)$ is compact .
	   Put 
	   \begin{align*}l := \ &\max \{ k\in \mathbb{N} \mid \text{there exist
	   points } p_i, \ i=1,\cdots ,k,\text{ such that}\\
	    &     \hspace{2.5cm} B_Y(p_i,a/2)\cap B_Y(p_j,a/2) =\emptyset
	    \text{ for any } i,j \text{ with } i\neq j              \}.
	    \end{align*}Then, there exist points $p_i$, $i=1, \cdots ,k,$
	   such that $B_Y (p_i,a/2) \cap B_Y (p_j,a/2) =\emptyset$ for
	   any $i,j$ with $i\neq j$. Hence, we get
	   \begin{align*}
	    1\geq \mu_Y(\bigcup_{i=1}^l
	    B_Y(p_i,a/2))=\sum_{i=1}^l\mu_Y(B_Y(p_i,a/2))=l\cdot v_Y(a/2).
	    \end{align*}It also follows from the definition of $l$ that
	   $\varphi_Y (T) \subseteq
	   \bigcup\limits_{i=1}^{l}B_Y(p_i,a)$. For any $i=1, \cdots
	   ,l$, we fix $t_i\in T$ with $p_i = \varphi_Y (t_i)$. 
	   \begin{claim}\label{haraitai}
	    \begin{align*}
	     \varphi_X(T)\subseteq
	     \bigcup\limits_{i=1}^lB_X \big(\varphi_X(t_i) ,a+c \big).
	     \end{align*}
	    \begin{proof}Take an arbitrary $q= \varphi_X(s) \in \varphi_X
	     (T)$, $s \in T$. Since $\varphi_Y(s)\in
	     \varphi_Y(T)\subseteq \bigcup\limits_{i=1}^l B_Y(p_i,a) $,
	     there exists $i$ with $1\leq i \leq l$ such that $\dist_Y
	     \big( \varphi_Y(s), p_i \big)\leq a$. Therefore, by using
	     $(2)$, we obtain 
	     \begin{align*}
	      \dist_X\big(  \varphi_X (s), \varphi_X(t_i)     \big)<
	      \dist_Y \big( \varphi_Y(s), p_i    \big) +c \leq a+c.
	      \end{align*}This completes the proof of the claim.
	     \end{proof}
	    \end{claim}
	   Applying Claim \ref{haraitai}, we get
	   \begin{align*}
	    1\leq \sum_{i=1}^{l}\frac{\mu_X
	    \big(B_X\big(\varphi_X(t_i),a+c\big)\big)}{\mu_X\big(\varphi_X(T)\big)}=l\cdot\frac{v_X(a+c)}{\mu_X\big(\varphi_X(T)\big)} 
\leq   \frac{v_X(a+c)}{v_Y(a/2)\cdot \mu_X\big(\varphi_X(T)\big)}.
	    \end{align*}
	   Since $ \mu_X \big( \varphi_X (T)      \big)\geq \mathcal{L}
	   \big( \varphi_X^{-1}\big( \varphi_X (T)\big)\big)\geq
	   \mathcal{L}(T)> 1-c$, we obtain
	   \begin{align*}
	    1\leq \frac{v_X(a+c)}{v_Y(a/2)\cdot
	    \mu_X\big(\varphi_X(T)\big)}<
	    \frac{v_X(a+c)}{v_Y(a/2)\cdot (1-c)}\leq 1,
	    \end{align*}which is a contradiction. This completes the
	   proof of Lemma \ref{ballnomeasure}.
	   \end{proof}
	  \end{lem}

	  For a compact Riemannian manifold $M$, we denote by $\vol(M)$ the
	  total Riemannian volume of $M$. We indicate by $\Gamma$ the Gamma function. 
	  \begin{lem}\label{hyouka}Let $M$ $($respectively, $N$$)$ be an $m$
	   $($respectively, $n$$)$-dimensional compact Riemannian
	   manifold having the uniformly distributed Riemannian volume
	   measure. Assume that $\ric_M \geq (m-1)\kappa_1 >0$ and
	   $\ric_N >0$, and put $a_N := \vol (N)/
	   \vol ({\mathbb{S}^n} )$. If a positive number $c$ with $c<1$
	   satisfies 
	   \begin{align*}
	    c^{n-m}\leq (1-c)    \frac{n a_N
	    (\kappa_1)^{m/2}\Gamma\big(\frac{m+1}{2} \big) \Gamma
	    \big( \frac{n}{2} \big)
	    }{m 2^{n+1}{\pi}^{m-1} \Gamma \big( \frac{m}{2}   \big)
	    \Gamma \big( \frac{n+1}{2}    \big)}  \text{ and } c
	    \sqrt{\kappa_1}\leq \pi, 
	    \end{align*}then we have $\sikakuu (M,N)\geq c$.
	   \begin{proof}From the Bishop-Gromov volume comparison
	    theorem, we get 
	    \begin{align*}
	     v_M (c/2)\geq v_{\mathbb{S}^m}\big((c\sqrt{\kappa_1})/2
	     \big)= \frac{\vol (\mathbb{S}^{m-1})}{\vol
	     (\mathbb{S}^m)}\int_{0}^{(c \sqrt{\kappa_1})/2}\sin^{m-1}
	     \theta d\theta.
	     \end{align*}From $c\sqrt{\kappa_1}\leq \pi$, we have $\sin
	    \theta \geq (\pi \theta) /2$ for any $\theta \in
	    [0,(c\sqrt{\kappa_1})/2] $. Hence, we obtain
	    \begin{align*}
	    v_M(c/2)\geq
	     \frac{2^{m-1}\vol (\mathbb{S}^{m-1})}{\pi^{m-1}\vol(\mathbb{S}^m)}
	     \int_0^{(c\sqrt{\kappa_1})/2}\theta^{m-1} d \theta =
	     \frac{c^m (\kappa_1)^{m/2}\vol (\mathbb{S}^{m-1})}{2m
	     \pi^{m-1} \vol (\mathbb{S}^m)}
	     \end{align*}
	    Let $\kappa_2 $ be a positive number such that $\ric_N \geq (n-1)\kappa_2$. We also obtain from
	    the Bishop inequality that 
	    \begin{align*}
	    v_N (2c) \leq \frac{v_{\mathbb{S}^n}(2c
	     \sqrt{\kappa_2})}{a_N (\kappa_2)^{n/2}}
	     =\frac{\vol (\mathbb{S}^{n-1})}{a_N 
	     (\kappa_2)^{n/2}\vol (\mathbb{S}^n)} \int^{2c\sqrt{\kappa_2}}_{0} \sin^{n-1}
	     \theta d\theta < \frac{(2c)^{n}\vol (\mathbb{S}^{n-1})}{na_N
	     \vol (\mathbb{S}^n)}. 
	     \end{align*}Recall that $\vol(\mathbb{S}^n)= 2 {\pi}^{(n+1)/2}/\Gamma
	   \big((n+1)/2\big)$. Therefore, combining above caluculations with
	    Lemma \ref{ballnomeasure}, we complete the proof.
	    \end{proof}
	   \end{lem}

	   \begin{proof}[Proof of Proposition \ref{kaotan}]Without loss
	    of generality, it may be assumed that $n_k \geq m_k$. 
	    
	    First, we consider the case of
	    $\{\mathbb{S}^n\}_{n=1}^{\infty}$. From the assumption, we
	    have $c^{n_k-m_k} \leq c^{C_3 k}$ for any
	    $0<c<1$. Substituting $n:=n_k$ and $m:=m_k$, we estimate the
	    right-hand side of the inequality of Lemma
	    \ref{hyouka} by 
	    \begin{align*}
	     (1-c)\frac{n_k \Gamma\big(  \frac{m_k+1}{2}    \big) \Gamma
	     \big(  \frac{n_k}{2}   \big)}{m_k 2^{n_k} \pi^{m_k-1}\Gamma
	     \big(   \frac{m_k}{2}   \big) \Gamma \big(
	     \frac{n_k+1}{2}\big)}\geq (1-c)2^{-C_1 k}\pi^{-C_2 k+1
	     }\frac{n_k \Gamma\big(  \frac{n_k}{2}     \big)}{m_k \Gamma
	     \big(  \frac{n_k +1}{2}         \big)}.
	     \end{align*}Therefore, if 
	    \begin{align*}
	    c \leq  \Big\{  (1-c)   \frac{n_k\Gamma \big( \frac{n_k}{2}
	     \big)}{m_k \Gamma \big(   \frac{n_k +1}{2}    \big)}
	     \Big\}^{\frac{1}{C_3
	     k}}2^{-\frac{C_1}{C_3}}\pi^{-\frac{C_2}{C_3}+\frac{1}{C_3k}},
	     \end{align*}then we obtain from Lemma \ref{hyouka} that
	    $\sikakuu (\mathbb{S}^{n_k}, \mathbb{S}^{m_k})\geq c$. Since
	    \begin{align*}
	     \Big\{  (1-c)   \frac{n_k\Gamma \big( \frac{n_k}{2}
	     \big)}{m_k \Gamma \big(   \frac{n_k +1}{2}
	    \big)}\Big\}^{\frac{1}{C_3 k}}\to 1\  \text{as }k\to \infty,
	     \end{align*}we have completed the proof
	    for $\{  \mathbb{S}^n   \}_{n=1}^{\infty}$.

	    Next, we consider $\{ \mathbb{C}P^n
	    \}_{n=1}^{\infty} $. It is well-known that
	    $\vol(\mathbb{C}P^n)=\pi^n /n! $ 
	    and the sectional curvature of
	    $\mathbb{C}P^n$ is bounded from below by $1$ (cf.~
	    \cite[Section 3.D.2, 3.H.3]{Gallot}). Hence, we get
	    \begin{align*}
	     a_{\mathbb{C}P^n} = \frac{\Gamma
	     \big(n+\frac{1}{2}\big)}{2\sqrt{\pi} n!}.
	     \end{align*}For any $0<c<1$, we have $c^{2n_k-2m_k}\leq
	    c^{2C_3 k}$. Substituting $n:=2n_k$ and $m:=2m_k$, we calculate
	    the right-hand side of the inequality of Lemma
	    \ref{hyouka} by
	    \begin{align*}
	     (1-c) \frac{\Gamma\big( m_k + \frac{1}{2}
	     \big)}{2\sqrt{\pi}m_k 2^{2 n_k +1} \pi^{2 m_k -1} \Gamma
	     (m_k)}\geq (1-c)\frac{1}{2\sqrt{\pi} C_2 k}\cdot 2^{-2C_1
	     k-1} \pi^{-2C_2 k +1}. 
	     \end{align*}So, if
	    \begin{align*}
	     c\leq \Big\{ (1-c)\frac{1}{2\sqrt{\pi} C_2 k}
	     \Big\}^{\frac{1}{2C_3 k}}2^{-\frac{C_1}{C_3}-\frac{1}{2 C_3
	     k}}\pi^{-\frac{C_2}{C_3} + \frac{1}{2 C_3 k}},
	     \end{align*}then we get by using Lemma \ref{hyouka} that
	    $\sikakuu(\mathbb{C}P^{n_k}, \mathbb{C}P^{m_k})\geq
	    c$. Since
	    \begin{align*}
	     \Big\{ (1-c)\frac{1}{2\sqrt{\pi} C_2 k}
	     \Big\}^{\frac{1}{2C_3 k}} \to 1 \text{ as } k\to \infty,
	     \end{align*}we complete the proof of the proposition. 
	    \end{proof}

Note that $\diam (\mathcal{X}_1, \sikakuu)=1$, where $\mathcal{X}_1$ is
the space of all mm-spaces with Borel probability measures.
	   \begin{lem}[{J. Christensen, c.f.~\cite[Section
	    $3.3$]{mattila}}]\label{christ}Let $X$ be a metric space and $\mu, \nu$
	    are uniformly distributed Borel measures on $X$. Then, there
	    exists a positive number $c>0$ such that $\mu = c \nu$.
	    \end{lem}

	    \begin{proof}[Proof of Proposition \ref{azusa}]We identify $\mathbb{S}^{n-1}$ with $\{(x_1,
	     \cdots ,x_n,0)\in \mathbb{S}^{n} \mid (x_1, \cdots ,x_{n})\in \mathbb{S}^{n-1}
	      \}$. Given an arbitrary $\varepsilon >0$, since the sequence
	     $\{ \mathbb{S}^n\}_{n=1}^{\infty}$ is a L\'{e}vy family, we
	     have $r_n:= \mu_n\big(
	      (\mathbb{S}^{n-1})_{\varepsilon}\big) \to 1$ as $n\to
	     \infty$. Hence, there is $m\in \mathbb{N}$ such that $1-r_n
	     < \varepsilon$ for any $n\geq m$. Suppose that $n\geq
	     m$. Taking two parameters $\Phi_1 :[0,r_n] \to
	     (\mathbb{S}^{n-1})_{\varepsilon}$ and $\Phi_2:(r_n,1]\to
	     \mathbb{S}^n \setminus (\mathbb{S}^{n-1})_{\varepsilon}$,
	     we define a Borel measurable map $\Phi:[0,1]\to \mathbb{S}^n$ by 
\begin{align*}
\Phi(t):=
\left\{
\begin{array}{ll}
\Phi_{1}(t) &  \ \ t\in [0,r_n], \\
{\Phi}_{2}(t) &  \ \ t\in (r_n,1]. \\
\end{array}
\right.
\end{align*}The map $\Phi$ is a parameter of $\mathbb{S}^n$. Let $\psi :\mathbb{S}^{n}\setminus \{(0,\cdots,0,1),(0, \cdots,0
	     ,-1)\}\to \mathbb{S}^{n-1}$ be the projection, that is,
	     $\psi(x)$ is the unique element of $\mathbb{S}^{n-1}$
	     satisfying $\dist_n\big(x,\psi(x)\big)= \dist_n(x,
	     \mathbb{S}^{n-1})$. Put $\varphi_1:= \psi \circ \Phi_1 :[0,r_n]\to
	     \mathbb{S}^{n-1}$. 
	     \begin{claim}$\varphi_{1 \ast}(\mathcal{L}) = r_n
	      \mu_{n-1}$.
	      \begin{proof}Take any Borel subset $A\subseteq
	       \mathbb{S}^{n-1}$. For any $g\in SO(n-1)$, we have
	       \begin{align*}
		\varphi_{1 \ast}(\mathcal{L})(gA)=r_n \mu_n
		\big(\psi^{-1}(gA) \big)=r_n \mu_n
		\big(\psi^{-1}(A) \big)= \varphi_{1 \ast}(\mathcal{L})(A).
		\end{align*}Hence, $\varphi_{1 \ast}(\mathcal{L})$
	       is a $SO(n-1)$-invariant Borel measure. From Lemma
	       \ref{christ}, we complete the proof of the claim. 
	       \end{proof}
	      \end{claim}Taking a parameter $\phi:(0,1]\to \mathbb{S}^{n-1}$
	     of $\mathbb{S}^{n-1}$, we define a Borel measurable map
	     $\varphi_2: (r_n,1]\to \mathbb{S}^{n-1}$ by
	     $\varphi_2(t):=\phi\big( (t-r_n)/(1-r_n)\big)$. Then, we
	     have $\varphi_{2 \ast}(\mathcal{L})=
	     (1-r_n)\mathcal{L}$. Therefore, defining a Borel measurable
	     map $\varphi:[0,1]\to \mathbb{S}^{n-1} $ by 
\begin{align*}
\varphi(t):=
\left\{
\begin{array}{ll}
\varphi_{1}(t) &  \ \ t\in [0,r_n], \\
\varphi_{2}(t) &  \ \ t\in (r_n,1], \\
\end{array}
\right.
\end{align*}we see that the map $\varphi$ is a parameter of
	     $\mathbb{S}^{n-1}$. Since
	     \begin{align*}&|\dist_n\big(\Phi (s), \Phi(t) \big) -
	      \dist_{n-1}\big( \varphi(s), \varphi(t)      \big)|\\=
	      \ & |\dist_n\big(\Phi_1(s), \Phi_1(t)\big)
	      -\dist_{n-1}\big( \varphi_1(s), \varphi_1(t)\big)|\leq
	      2\varepsilon
	      \end{align*}for any $s,t \in [0,r_n]$, we get
	     \begin{align*}
	      \sikakuu (\mathbb{S}^n, \mathbb{S}^{n-1})\leq \square_1
	      (\Phi^{\ast} \dist_n , \varphi^{\ast} \dist_{n-1})\leq
	      \max \{ 2\varepsilon, 1-r_n       \}=2\varepsilon.
	      \end{align*}Consequently, we obtain $\sikakuu
	     (\mathbb{S}^n, \mathbb{S}^{n-1})\to 0$ as $n\to
	     \infty$. A similar argument shows that $\sikakuu
	     (\mathbb{C}P^n , \mathbb{C}P^{n-1})\to 0$ as $n\to
	     \infty$. This completes the proof of Proposition \ref{azusa}.
	     \end{proof}

	 \begin{lem}\label{asobisugi}For any $n,m \in \mathbb{N}$, we have
	  \begin{align*}
	   \sikakuu \big(  SO(n), SO(m)   \big) \geq c(n,m):= \min \Big\{
	   \frac{1}{2},       |\diam SO(n) -\diam SO(m)| \Big\}.
	   \end{align*}
	   \begin{proof}Suppose that $n>m$ and $\sikakuu \big(  SO(n),
	    SO(m)   \big) < c(n,m)$. There exist compact subset $T
	    \subseteq [0,1]$ and two parameters $\varphi_n:[0,1]\to
	    SO(n)$, $\varphi_m:[0,1]\to SO(m)$ such that
	    \begin{itemize}
	     \item[$(1)$]$\mathcal{L}(T)> 1-c(n,m)\geq 1/2$, 
	     \item[$(2)$] $\varphi_n |_T
	    :T \to SO(n)$, $\varphi_m |_{T} :T\to SO(m)$ are
	    continuous,
	     \item[$(3)$] $|\dist_{n}\big( \varphi_n(s) ,
	    \varphi_n(t)   - \dist_{m} \big(  \varphi_m(s),
	    \varphi_m(t)          \big)            |< c(n,m)$ for any $s, t
	    \in T$.
	     \end{itemize} 
	    \begin{claim}\label{peach}There exist $s_0, t_0\in T$ such that
	     $\dist_{n}\big( \varphi_n(s_0), \varphi_n(t_0)    \big)=
	     \diam SO(n)$.
	     \begin{proof}
	      Take $A_0, B_0 \in SO(n)$ such that $\diam SO(n)=
	      \dist_{n} (A_0, B_0)$ and define a map $\psi :SO(n) \to
	      SO(n)$ by $\psi (A):= A A_0^{-1}B$. Then,
	      $\psi_{\ast} (\mu_n) =\mu_n$ and $\dist_{n}
	      \big(A, \psi(A)\big)= \diam SO(n)$ for any $A \in
	      SO(n)$. Suppose that $\dist_n \big( \varphi_n(s) ,
	      \varphi_n(t)         \big) < \diam SO(n)$ for any $s,t \in
	      T$. Then, we get $\psi \big(  \varphi_n(T)      \big) \cap
	      \varphi_n (T)= \emptyset$, which leads to
	      \begin{align*}
	       \mu_n   \big( \psi \big(  \varphi_n(T)          \big)
	       \cap \varphi_n (T ) \big)= \ &\mu_n \big( \psi \big(
	       \varphi_n(T)          \big)            \big) + \mu_n
	       \big(  \varphi_n (T)    \big)\\
	       = \ & \mu_n \big( \psi^{-1} \big(   \psi\big( \varphi_n
	       (T)\big) \big)  \big) + \mu_n \big(  \varphi_n(T)
	       \big)\\
	       \geq \ &2 \mu_n \big(  \varphi_n(T)      \big) > 1.
	       \end{align*}This is a contradiciton and thus we complete
	      the proof of the claim. 
	      \end{proof}
	     \end{claim}
	    By Claim \ref{peach}, we obtain
	    \begin{align*}
	     \diam SO(n)- \diam SO(m)\leq |\dist_n \big(  \varphi_n
	     (s_0),       \varphi_n(t_0)\big)    - \dist_m\big(  \varphi_m(s_0),      \varphi_m(t_0) \big) |<c(n,m),
	     \end{align*}which is a contradiciton. This completes
	    the proof of Lemma \ref{asobisugi}.
	   \end{proof}
	  \end{lem}

	   \begin{proof}[Proof of Proposition \ref{oosawaakane}]An easy caluculations show that $2
	    \sqrt{n-1}\leq \diam SO(n) \leq 2\sqrt{n}$. Therefore,
	    supposing $n_k\geq m_k$, we
	    have
	    \begin{align*}
	     \diam SO(n_k)-\diam SO(m_k)\geq \ & 2\sqrt{n_k - 1} -
	     2\sqrt{m_k}\\= \ & 2\frac{n_k-
	     m_k-1}{\sqrt{n_k-1}+\sqrt{m_k} }\geq 2 \frac{C_3 -1/\sqrt{k}}{\sqrt{C_1-1/k}+ \sqrt{C_2}}.
	     \end{align*}Thus, applying Lemma \ref{asobisugi}, we complete
	    the proof.
	    \end{proof}

	\begin{ack}\upshape
	 The author would like to thank Professor Takashi Shioya and
	 Mr. Masayoshi Watanabe for
	 valuable discussions and many fruitful suggestions. He also
	 thanks Proffesor Hajime Urakawa for his encouragement.
	 \end{ack}

	\end{document}